\documentclass[11pt,a4paper,reqno]{amsart}
\usepackage[english]{babel}
\usepackage[applemac]{inputenc}
\usepackage[T1]{fontenc}
\usepackage{palatino}
\usepackage{cite}
\usepackage{esint}
\usepackage{mathabx}
\usepackage{verbatim}
\usepackage{amsmath}
\usepackage{amssymb}
\usepackage{amsthm}
\usepackage{amsfonts}
\usepackage{graphicx}
\usepackage{mathtools}
\usepackage{enumitem}

\usepackage[colorlinks = true, citecolor = black]{hyperref}
\title[Quantitative Alberti representations]{Quantitative Alberti representations \\ in spaces of bounded geometry}
\keywords{Alberti representations, spaces of $Q$-bounded geometry}
\author{Tuomas Orponen}
\address{University of Helsinki, Department of Mathematics and Statistics}
\subjclass[2010]{28A50 (Primary) 30LXX, 28A78 (Secondary)}
\thanks{T.O. is supported by the Academy of Finland via the project \emph{Quantitative rectifiability in Euclidean and non-Euclidean spaces}, grant No. 309365.}
\email{tuomas.orponen@helsinki.fi}

\newcommand{\R}{\mathbb{R}}

\newcommand{\N}{\mathbb{N}}

\newcommand{\C}{\mathbb{C}}

\newcommand{\Z}{\mathbb{Z}}
\newcommand{\tn}{\mathbb{P}}

\newcommand{\calD}{\mathcal{D}}
\newcommand{\calH}{\mathcal{H}}

\newcommand{\calR}{\mathcal{R}}

\newcommand{\spt}{\operatorname{spt}}

\newcommand{\calP}{\mathcal{P}}

\newcommand{\spa}{\operatorname{span}}
\newcommand{\diam}{\operatorname{diam}}
\newcommand{\card}{\operatorname{card}}
\newcommand{\dist}{\operatorname{dist}}

\renewcommand{\mod}{\mathrm{mod}}
\newcommand{\Lip}{\mathrm{Lip}}

\numberwithin{equation}{section}

\theoremstyle{plain}
\newtheorem{thm}[equation]{Theorem}

\newtheorem{lemma}[equation]{Lemma}

\newtheorem{cor}[equation]{Corollary}
\newtheorem{proposition}[equation]{Proposition}
\newtheorem{question}{Question}

\theoremstyle{definition}

\newtheorem{definition}[equation]{Definition}

\theoremstyle{remark}

\newcommand{\nref}[1]{(\hyperref[#1]{#1})}

\begin{document}

\begin{abstract} A metric measure space $(X,d,\mu)$ is said to be \emph{$A_{\infty}$ on curves} if there exist constants $\tau < 1$ and $\theta > 0$ with the following property. For every $x \in X$, $0 < r \leq \diam(X)$, and a Borel set $S \subset B(x,r)$ with $\mu(S) > \tau \mu(B(x,r))$, there exists a continuum $\gamma \subset X$ of length $\leq r$ satisfying $\calH^{1}_{\infty}(\gamma \cap S) \geq \theta r$. 

I first observe that spaces of $Q$-bounded geometry, $Q > 1$, are $A_{\infty}$ on curves. Then, I show that any complete, doubling, and quasiconvex space $(X,d,\mu)$ which is $A_{\infty}$ on curves has Alberti representations with $L^{p}$-densities for some $p > 1$, depending only on the doubling and $A_{\infty}$-constants. More precisely, any normalised restriction of $\mu$ to a ball $B \subset X$ can be written as $\mu_{B} = f_{B} \, d\nu_{B}$, where $\nu_{B}$ is a convex combination of measures of linear growth supported on continua of length $\leq \diam(B)$, and $\|f_{B}\|_{L^{p}(\nu_{B})} \leq C$ for some constant $C \geq 1$ independent of $B$.
\end{abstract}

\maketitle

\tableofcontents

\section{Introduction}

Let $Q > 1$, and consider a space $(X,d,\mu)$ of \emph{$Q$-bounded geometry} (see \cite{MR1869604} for the origin of this terminology): $(X,d)$ is proper and path connected, $\mu$ is Borel and $Q$-regular, that is,
\begin{displaymath} \mu(B(x,r)) \sim r^{Q}, \qquad x \in X, \: 0 < r \leq \diam(X), \end{displaymath}
and $X$ supports a weak $(1,Q)$-Poincar\'e inequality: there exist constants $C,\lambda \geq 1$ such that whenever $x \in X$, $r > 0$, $u \in L^{1}_{\mathrm{loc}}(\mu)$, and $g \geq 0$ is an upper gradient of $u$, then
\begin{equation}\label{poincare} \fint_{B(x,r)} |u - u_{B(x,r)}| \, d\mu \leq Cr \left(\fint_{B(x,\lambda r)} g^{Q} \, d\mu \right)^{1/Q}. \end{equation} 
Recall that a Borel function $g \colon X \to [0,\infty]$ is an upper gradient of $u$ if 
\begin{displaymath} |u(x) - u(y)| \leq \int_{\gamma} g \, d\calH^{1} \end{displaymath}
for all rectifiable continua $\gamma \subset X$ joining $x$ to $y$. Recall that a \emph{continuum} is a compact and connected (subset of a) metric space, and a \emph{rectifiable continuum} is a continuum of finite $\calH^{1}$-measure. For more information on Poincar\'e inequalities and upper gradients, see for example \cite{MR3363168}. The properties above imply that $(X,d)$ is quasiconvex, see \cite[Proposition 4.4]{MR1683160}. Also, a space $(X,d,\mu)$ with the properties above is \emph{$Q$-Loewner}, see \cite[Theorem 5.7]{HK}: there is a function $\psi \colon (0,\infty) \to (0,\infty)$ so that whenever $E,F \subset X$ are non-degenerate continua satisfying $\dist(E,F) \leq t\min\{\diam(E),\diam(F)\}$, then
\begin{equation}\label{loewner} \mod_{Q}(E,F) \geq \psi(t). \end{equation}
The Loewner property is a way of formalising that $X$ contains plenty of rectifiable continua: by definition, \eqref{loewner} means that whenever $g \colon X \to [0,\infty]$ is a Borel function satisfying
\begin{displaymath} \int_{\gamma} g \, d\calH^{1} \geq 1 \end{displaymath} 
for all continua $\gamma$ intersecting both $E$ and $F$, then 
\begin{displaymath} \int g^{Q} \, d\mu \geq \psi(t). \end{displaymath}
Informally, one could summarise that the $L^{Q}(\mu)$-norm of any Borel function $g \colon X \to \R$ is dominated from below by the $L^{1}(\calH^{1}\llcorner \gamma)$-norm of $g$ for some (not-too-short) rectifiable continuum $\gamma \subset X$. In this paper, we investigate the possibility of a converse statement. Here is a (rather inaccurate) guiding question:
\begin{question}\label{mainQ} Is the $L^{1}(\mu)$-norm of any Borel function $g \colon X \to \R$ dominated from above by the $L^{p}(\calH^{1}\llcorner \gamma)$-norm of $g$ for some (not-too-long) rectifiable curve $\gamma \subset X$, for some $p < \infty$? \end{question}

Qualitative statements of this type have appeared in the literature on several occasions. For example, Shanmugalingam proves in \cite[Lemma 6.2]{MR1809341} (see also \cite[Theorem 6.2]{MR1809341}) that, under the hypotheses on $X$ above, any subset $E \subset X$ avoiding $\mod_{Q}$ almost every continuum has $\mu(E) = 0$. An even more closely related statement is \cite[Theorem 6.6]{MR3300699}, which implies in particular that if $(X,d,\mu)$ is a Lipschitz differentiability space, then any set $S \subset X$ intersecting every rectifiable continuum in $\calH^{1}$-measure zero has $\mu(S) = 0$. For illustrative purposes, we now include the following simple proposition:

\begin{proposition}\label{amuseBouche} Let $(X,d,\mu)$ be a complete and doubling metric measure space admitting the weak Poincar\'e $(1,Q)$-inequality \eqref{poincare} for some $Q \geq 1$. If $S \subset X$ is a Borel set with $\calH^{1}(S \cap \gamma) = 0$ for all rectifiable continua $\gamma \subset X$, then $\mu(S) = 0$. \end{proposition}
\begin{proof} By \cite[Corollary 8.3.16]{MR3363168}, the space $X$ admits a geodesic metric equivalent to $d$; hence we may assume that $(X,d)$ is geodesic to begin with. Fix $x_{0} \in X$, and consider the $1$-Lipschitz function $f(x) := d(x,x_{0})$. Then
\begin{displaymath} \mathrm{Lip}(f,x) := \mathop{\limsup_{y \to x}}_{y \neq x} \frac{|f(x) - f(y)|}{d(x,y)} = 1, \qquad x \in X. \end{displaymath}  
Indeed, take a curve $\gamma$ connecting $x$ to $x_{0}$ with $\calH^{1}(\gamma) = d(x,x_{0})$. If $y \in \gamma$, then 
\begin{displaymath} d(x,x_{0}) = \calH^{1}(\gamma) = \calH^{1}(\gamma\llcorner_{[x,y]}) + \calH^{1}(\gamma\llcorner_{[y,x_{0}]}) \geq d(x,y) + d(y,x_{0}), \end{displaymath} 
which implies that $f(x) - f(y) = d(x,y)$. Now, if $S \subset X$ is a Borel set which intersects all rectifiable continua in zero length, then evidently $g := \mathbf{1}_{S^{c}}$ is an upper gradient for $f$. Moreover, by \cite[Proposition 4.3.3]{MR2041901}, 
\begin{displaymath}  \Lip(f,x) \lesssim \limsup_{r \to 0} \frac{1}{r} \fint_{B(x,r)} |f - f_{B(x,r)}| \, d\mu \end{displaymath} 
for $\mu$ almost every $x \in X$. Hence, using the Poincar\'e inequality,
\begin{align} 1 = \Lip(f,x) & \lesssim \limsup_{r \to 0} \left( \fint_{B(x,\lambda r)} g^{Q} \, d\mu \right)^{1/Q} \notag \\
&\label{form26} = \limsup_{r \to 0} \left( \frac{\mu(B(x, r) \cap S^{c})}{\mu(B(x, r))} \right)^{1/Q} \end{align} 
for $\mu$ almost every $x \in X$. In particular, this holds for $\mu$ almost every $x \in S$. However, by Lebesgue differentiation, see \cite[Section 3.4]{MR3363168}, the quantity in \eqref{form26} equals zero for $\mu$ almost every $x \in S$. We conclude that $\mu(S) = 0$, as desired. \end{proof}

To answer Question \ref{mainQ}, we will need to understand what happens if $\calH^{1}(S \cap \gamma)$ is not zero, but only "small" for rectifiable continua $\gamma \subset X$. The argument above appears to be useless in this situation (at least $\mathbf{1}_{S^{c}}$ no longer seems to be an upper gradient for anything), so we need to find a different approach. We begin with a definition. 

\begin{definition}[$A_{\infty}(\Gamma)$-space]\label{Ainfty} Let $(X,d,\mu)$ be a metric measure space. We say that $X$ is \emph{$A_{\infty}$ on curves}, or that $X$ is an \emph{$A_{\infty}(\Gamma)$-space}, if there are constants  $\tau < 1$ and $\theta > 0$ with the following property. Whenever $x \in X$, $0 < r \leq \diam(X)$, and $S \subset B(x,r)$ is a Borel set with $\mu(S) > \tau \mu(B(x,r))$, there exists a continuum $\gamma \subset B(x,r)$ such that
\begin{displaymath} \calH^{1}(\gamma) \leq r \quad \text{and} \quad \calH^{1}_{\infty}(\gamma \cap S) \geq \theta r. \end{displaymath}  
\end{definition}

This quantification is weak: it only says that if a subset of an $r$-ball has sufficiently short intersections with all continua of length $\leq r$, then the subset cannot occupy all of the ball. On a positive note, this property follows easily from the $Q$-Loewner condition:
\begin{proposition}\label{Qinfty} Every space of $Q$-bounded geometry, $Q > 1$, is $A_{\infty}$ on curves. \end{proposition}
For a proof, see Section \ref{AinftyAndPoincare}. It would be interesting to know if weaker hypotheses (than $Q$-bounded geometry) imply the $A_{\infty}(\Gamma)$-property; in the present paper, the $Q$-bounded geometry assumption only appears in the proposition above, while the subsequent results just rely on the $A_{\infty}(\Gamma)$-condition combined with milder \emph{a priori} axioms (completeness, doubling, and quasiconvexity).

It turns out that the $A_{\infty}(\Gamma)$-property allows for a very substantial amount of self-improvement: it implies a positive answer to Question \ref{mainQ}. To make this precise, we consider the following notion of \emph{$1$-rectifiable representations}. These entities are also known as \emph{Alberti representations}, as they first appeared in the work of Alberti \cite{MR1215412} on the rank-$1$ property for BV functions. The definition below is a slight variant of the ones used in \cite{MR3300699,MR3494485,2019arXiv190400808B}. In this paper, a \emph{$1$-Frostman measure} is a Borel measure $\nu$ on some metric space $Y$ satisfying $\nu(B(x,r)) \leq r$ for all balls $B(x,r) \subset Y$. 
\begin{definition} Let $(Y,d)$ be a compact metric space, let $r > 0$, and let $\mathcal{P}_{r}$ be the set of Borel probability measures on the compact space (see Lemma \ref{compactnessLemma}) of measures of linear growth supported on continua $\gamma \subset Y$ of length at most $r$. Let $\nu_{\tn}$ be the measure
\begin{displaymath} \nu_{\tn} := \frac{1}{r} \int \nu \, d\tn(\nu). \end{displaymath} 
A Borel measure $\mu$ on $Y$ has a \emph{$1$-rectifiable representation of length $r$ in $L^{p}$} if there exists $\tn \in \mathcal{P}_{r}$ such that $\mu \ll \nu_{\tn}$ with $\mu \cong d\mu/d\nu_{\tn} \in L^{p}(\nu_{\tn})$. In this case, we write
\begin{displaymath} \|\mu\|_{L^{p}(r)} := \inf\{\|\mu\|_{L^{p}(\nu_{\tn})} : \tn \in \mathcal{P}\} < \infty. \end{displaymath} 
\end{definition}
Now we come to the main theorem of the paper.
\begin{thm}\label{main} Let $(X,d,\mu)$ be a complete, doubling, and quasiconvex space which is $A_{\infty}$ on curves. Then, there exist constants $p > 1$ and $A \geq 1$ such that the following holds. For any $x \in X$ and $0 < r \leq \diam(X)$, the normalised restriction 
\begin{displaymath} \mu_{x,r} := \frac{\mu\llcorner_{B(x,r)}}{\mu(B(x,r))} \end{displaymath}
has a $1$-rectifiable representation of length $r$ in $L^{p}$, and moreover $\|\mu_{x,r}\|_{L^{p}(r)} \leq A$. \end{thm}
The main message here is that $p > 1$, because the theorem then implies the following answer to Question \ref{mainQ}:
\begin{cor}\label{mainCorIntro} Same assumptions and notation as in Theorem \ref{main}. Let $p^{\ast} := p/(p - 1) < \infty$ be the dual exponent of the parameter $p > 1$ found in Theorem \ref{main}. Let $x \in X$, $0 < r \leq \diam(X)$, and $B := B(x,r)$. Then, for any bounded Borel function $g \colon X \to \R$, we have
\begin{equation}\label{inverseLoewner} \frac{1}{\mu(B)} \int_{B} |g| \, d\mu \lesssim \sup_{\gamma} \left(\frac{1}{r} \int_{\gamma} |g|^{p^{\ast}} \, d\calH^{1} \right)^{1/p^{\ast}}, \end{equation} 
where the $\sup$ runs over continua of length $\leq r$ contained in $\overline{2B}$.
\end{cor}
In particular, by Theorem \ref{Qinfty}, spaces of $Q$-bounded geometry satisfy the "inverse Loewner property" \eqref{inverseLoewner} with some finite exponent. The proof of Theorem \ref{main} is given at the end Section \ref{s:representations}. It turns out that Corollary \ref{mainCorIntro} is more elementary than Theorem \ref{main} -- and is indeed a crucial step in the proof of Theorem \ref{main}! So, although Corollary \ref{mainCorIntro} easily follows from Theorem \ref{main}, we also need to give a separate proof for it, see Corollary \ref{mainCor}.

\subsection*{Acknowledgements} I would like to thank David Bate and Katrin F\"assler for useful discussions, and for pointing out references.

\section{Spaces of $Q$-bounded geometry are $A_{\infty}$ on curves}\label{AinftyAndPoincare}

This section contains the proof of Proposition \ref{Qinfty}. Before that, we record the well-known fact that rectifiable continua admit parametrisations by Lipschitz functions defined on compact intervals:
\begin{proposition}\label{parametrisation} Let $(\Gamma,d)$ be a rectifiable continuum. Then, there exists an absolute constant $C \geq 1$ and a surjective Lipschitz map $\gamma \colon [0,1] \to \Gamma$ with Lipschitz constant at most $C\calH^{1}_{d}(\Gamma)$. 
\end{proposition}

\begin{proof} The proposition is proven in \cite[Lemma 3.7]{MR2373273} with the constant $C = 32$ under the assumption that $\Gamma$ is contained in some Hilbert space $H$. It would, however, appear that the proof works without modification if $\Gamma$ is contained in a Banach space instead, and this can always be achieved by Kuratowski embedding. 

Another, more recent, reference is \cite[Theorem 4.4]{MR3614660} which also contains more precise information, and yields the statement with constant $C = 2$. \end{proof}

In the present paper, the lemma above is only needed to infer the following: if $\Gamma$ is a rectifiable continuum, $\nu$ is an outer measure supported on $\Gamma$, and $S \subset \Gamma$ is arbitrary, then there exists a sub-continuum $\Gamma' \subset \Gamma$ of length $\leq \mathcal{H}^{1}(\Gamma)/2$ with $\nu(S \cap \Gamma') \gtrsim \nu(S \cap \Gamma)$. This is clear once $\Gamma$ is parametrised as $\Gamma = \gamma([0,1])$, where $\gamma \colon [0,1] \to \Gamma$ is $C\calH^{1}(\Gamma)$-Lipschitz: then, by the pigeonhole principle, and the sub-additivity of $\nu$, the choice $\Gamma' := \gamma(I)$ will work for some interval $I \subset [0,1]$ of length $1/(2C)$.

We then prove proposition \ref{Qinfty}.

\begin{proposition}\label{AinftyPoincare} Let $(X,d,\mu)$ be a space of $Q$-bounded geometry for some $Q > 1$. Then $X$ is $A_{\infty}$ on curves.
\end{proposition}

\begin{proof} The proof of \cite[Lemma 3.17]{HK} shows that for every $\delta \in (0,1)$ there exist constants $C = C(\delta) \geq 1$ and $\epsilon = \epsilon(\delta) > 0$ such that if $E,F \subset B(x,r)$ are disjoint continua of diameter $\geq \delta r$, then $\mod_{Q}(E,F;B(x,Cr)) \geq \epsilon$. Moreover, by the $Q$-regularity and quasiconvexity of $X$, there exists $\delta > 0$ such that every ball $B(x,r)$ with $0 < r \leq \diam(X)$ contains two continua of diameter $\geq \delta r$ with pairwise distance $\geq \delta r$. Let $C,\epsilon > 0$ be the constants, as above, corresponding to this $\delta$.

Now, to prove the proposition, fix $x \in X$ and $0 < r \leq \diam(X)$. Choose two disjoint continua in $B(x,r/C)$ of diameter $\geq \delta r/C$ and
\begin{equation}\label{distEst} \dist(E,F) \geq \delta r/C. \end{equation}
Then
\begin{equation}\label{modulus1} \mod_{Q}(E,F;B(x,r)) \geq \epsilon, \end{equation}
Further, since
\begin{displaymath} \mod_{Q}(\{\gamma \subset B(x,r) : \calH^{1}(\gamma) \geq Ar\}) \lesssim A^{-Q}, \qquad A \geq 1, \end{displaymath}
see \cite[Lemma 3.15]{HK}, the modulus estimate \eqref{modulus1} remains essentially valid if we restrict to continua connecting $E$ to $F$ inside $B(x,r)$ of length $\leq Ar$:
\begin{equation}\label{modulus2} \mod_{Q}(E,F;B(x,r),\leq Ar) \geq \epsilon/2. \end{equation}
Now, to prove the $A_{\infty}(\Gamma)$-property of $X$, let $\theta > 0$ be a small constant, and let $S \subset B(x,r)$ be a Borel set satisfying $\calH^{1}_{\infty}(S \cap \gamma) \leq \theta r$ for all continua $\gamma \subset B(x,r)$ with $\calH^{1}(\gamma) \leq r$. It immediately follows that $\calH^{1}_{\infty}(S \cap \gamma) \lesssim A\theta r$ for all continua $\gamma \subset B(x,r)$ with $\calH^{1}(\gamma) \leq Ar$ (use Proposition \ref{parametrisation}). In particular 
\begin{equation}\label{form28} \calH_{\infty}^{1}(S \cap \gamma) \leq \frac{\dist(E,F)}{2} \end{equation}
for all such continua if $0 < \theta \ll \delta/(2AC)$, recalling \eqref{distEst}. Consequently, the function
\begin{displaymath} \rho := c \cdot \frac{\mathbf{1}_{B(x,r) \cap S^{c}}}{r}  \end{displaymath} 
is admissible for the family in \eqref{modulus2}, if $c \geq 2C/\delta$, recall \eqref{distEst}. Indeed, if $\gamma \subset B(x,r)$ meets both $E$ and $F$, then $\calH_{\infty}^{1}(B(x,r) \cap \gamma) \geq \dist(E,F)$, and consequently
\begin{displaymath} \calH^{1}(\gamma \cap [B(x,r) \cap S^{c}]) \geq \calH_{\infty}^{1}(\gamma \cap [B(x,r) \cap S^{c}]) \stackrel{\eqref{form28}}{\geq} \frac{\dist(E,F)}{2} \geq \frac{\delta r}{2C} \end{displaymath} 
by the sub-additivity of Hausdorff content. Now, we may infer from \eqref{modulus2} that
\begin{displaymath} \frac{\epsilon}{2} \leq \int \rho^{Q} \, d\mu = c^{Q} \cdot \frac{\mu(B(x,r) \cap S^{c})}{r^{Q}}, \end{displaymath} 
whence
\begin{displaymath} \mu(B(x,r) \cap S^{c}) \gtrsim r^{Q} \sim \mu(B(x,r)). \end{displaymath} 
This implies that $\mu(B(x,r) \cap S) \leq \tau \mu(B(x,r))$ for some $\tau < 1$ independent of $x$ and $r$. Consequently $X$ is $A_{\infty}$ on curves and the proof of the proposition is complete.  \end{proof}

\begin{lemma} Let $(X,d,\mu)$ be a complete quasiconvex metric measure space which is $A_{\infty}$ on curves. Then, there exist constants $\tau < 1$ and $A,\theta > 0$ with the following property. \end{lemma}

\section{Self-improvement of the $A_{\infty}(\Gamma)$-property}

In this section, $(X,d,\mu)$ is a complete and doubling space. It will be convenient to use a system of dyadic (Christ) cubes on $X$. This refers to a collection $\calD = \bigcup_{j \in \Z} \calD_{j}$ of open sets with the following properties:
\begin{itemize}
\item[(D1) \phantomsection \label{D1}] $\mu(X \setminus \cup \calD_{j}) = 0$ for all $j \in \Z$.
\item[(D2) \phantomsection \label{D2}] Each $Q \in \calD_{j}$ has \emph{side-length} $\ell(Q) := 2^{-j}$ and satisfies $\diam(Q) \leq \ell(Q)$.
\item[(D3) \phantomsection \label{D3}] Each $Q \in \calD$ contains a ball $B_{Q} \subset Q$ of radius $\ell(Q) \lesssim \mathrm{rad}(B_{Q}) \leq \ell(Q)$.
\item[(D4) \phantomsection \label{D4}] Each $Q \in \calD$ has a "small boundary" in the following sense: there are constants $C \geq 1$ and $\eta > 0$ (independent of $Q$) such that
\begin{displaymath} \mu(\{x \in Q : \dist(x,X \setminus Q) \leq \epsilon \ell(Q)\}) \leq C\epsilon^{\eta}\mu(Q), \qquad \epsilon > 0. \end{displaymath} 
\end{itemize}
Such a collection $\calD$ exists by \cite[Theorem 11]{Christ}. The next lemma says that the the $A_{\infty}(\Gamma)$-property can also be formulated in terms of dyadic cubes.

\begin{lemma}\label{prop1} Let $(X,d,\mu)$ be a complete and doubling $A_{\infty}(\Gamma)$-space. Then, there exist constants $\tau < 1$ and $\theta > 0$ such that the following holds. Let $Q \in \calD$ be a cube, and let $S \subset Q$ be a Borel set with $\mu(S) \geq \tau \mu(Q)$. Then, there exists a continuum $\gamma \subset Q$ of length $\calH^{1}(\gamma) \leq \ell(Q)$ such that $\calH^{1}_{\infty}(S \cap \gamma) \geq \theta \ell(Q)$.
\end{lemma}

\begin{proof} Recall from \nref{D3} that each $Q \in \calD$ contains a ball $B_{Q}$ of radius comparable to $\ell(Q)$; by taking a smaller ball inside $B_{Q}$ if needed, we will here assume that the radius of $B_{Q}$ is only a small fraction of $\ell(Q)$. If now $S \subset Q$ is a Borel set with $\mu(S) \geq \tau \mu(Q)$, then
\begin{displaymath} \mu(B_{Q} \, \setminus \, S) \leq \mu(Q \, \setminus \, S) \leq (1 - \tau)\mu(Q) \lesssim (1 - \tau)\mu(B_{Q}). \end{displaymath}
This implies that the ratio $\mu(S \cap B_{Q})/\mu(B_{Q})$ can be made arbitrarily close to one by taking $\tau$ close enough to one. Consequently, the $A_{\infty}(\Gamma)$-condition (for balls) will be applicable to the Borel subset $S \cap B_{Q}$ of the ball $B_{Q}$. The conclusion is that there is a continuum $\gamma \subset B_{Q} \subset Q$ of length $\lesssim \mathrm{diam}(B_{Q}) \leq \ell(Q)$ satisfying $\calH_{\infty}^{1}(S \cap \gamma) \geq \calH^{1}_{\infty}(S \cap B_{Q} \cap \gamma) \gtrsim \mathrm{rad}(B_{Q})$. This implies the statement. \end{proof}

The following proposition states that the $A_{\infty}(\Gamma)$-property implies a stronger version of itself. Whereas the $A_{\infty}(\Gamma)$-condition says that any subset of $B(x,r)$ with short intersection with every curve "cannot be all of $B(x,r)$", the proposition says, \emph{a fortiori}, that any such subset of $B(x,r)$ can only occupy a small fraction of $B(x,r)$.  
\begin{proposition}\label{mainProp} Let $(X,d,\mu)$ be a complete, doubling, and quasiconvex $A_{\infty}(\Gamma)$-space. Then, there exist constants $C \geq 1$ and $1 \leq q < \infty$ such that the following holds. Let $Q \in \calD$, and let $S \subset Q$ be a Borel set. Then, there exists a continuum $\gamma \subset X$ of length $\calH^{1}(\gamma) \leq \ell(Q)$ such that 
\begin{equation}\label{form23} \frac{\mu(S)}{\mu(Q)} \leq C\left(\frac{\calH_{\infty}^{1}(S \cap \gamma)}{\ell(Q)} \right)^{1/q}. \end{equation}
\end{proposition}

\begin{proof} We will prove the following by induction on $n \geq 1$. Let $\theta > 0$ and $\tau < 1$ be the parameters from Lemma \ref{prop1}, let $\rho \in (\tau,1)$ (for example $\rho = (1 + \tau)/2$ will do), and let $c > 0$ be another small enough constant. If $Q \in \calD$ and $S \subset Q$ is a Borel set with 
\begin{equation}\label{form6} \rho^{n}\mu(Q) < \mu(S) \leq \rho^{n - 1}\mu(Q), \end{equation}
then there exists a continuum $\gamma \subset X$ satisfying
\begin{equation}\label{form15} \calH^{1}(\gamma) \leq \ell(Q) \quad \text{and} \quad \calH^{1}_{\infty}(S \cap \gamma) \geq (c\theta)^{n}\ell(Q), \end{equation}
where $c \in (0,1]$ is a suitable small constant, depending on the doubling constant of $X$, at least so small that $c\theta \leq \rho$. Let us briefly argue that this claim will prove the proposition. If $Q \in \calD$ and $S \subset Q$ is a Borel set with $\mu(S) > 0$, then there exists $n \geq 1$ such that \eqref{form6} holds. Then, by \eqref{form15}, there will exist $\gamma \subset X$ with $\calH^{1}(\gamma) \leq \ell(Q)$ and
\begin{displaymath} \frac{\calH^{1}_{\infty}(S \cap \gamma)}{\ell(Q)} \geq (c\theta)^{n} = \rho^{n \log_{\rho} (c\theta)} \gtrsim \left(\frac{\mu(S)}{\mu(Q)} \right)^{\log_{\rho} (c\theta)}. \end{displaymath} 
Here $1 \leq q := \log_{\rho} (c\theta) < \infty$, so the inequality above implies \eqref{form23}.

We then begin the induction. The case $n = 1$ follows immediately from Proposition \ref{prop1}, since $\rho \geq \tau$ and $c \leq 1$. So, we assume that $n \geq 2$, and the claim has already been established for index $n - 1$, and for all cubes $Q \in \calD$. Let then $Q \in \calD$, and $S \subset Q$ be a Borel set satisfying \eqref{form6}. Let $\calP$ be collection of maximal sub-cubes of $Q$ satisfying 
\begin{equation}\label{form8} \mu(P \cap S) \geq \tau \mu(P), \qquad P \in \calP. \end{equation} 
Note that $\mu$ almost all of $S$ is contained in the union of the cubes in $\calP$ by Lebesgue differentiation. If $Q \in \mathcal{P}$, then we are in the situation of Proposition \ref{prop1} and can certainly choose a continuum $\gamma \subset Q \subset X$ satisfying \eqref{form15}. So, we may assume that $Q \notin \mathcal{P}$. Therefore the parents $\widehat{P}$ of the cubes $P \in \calP$ are contained in $Q$, and satisfy
\begin{equation}\label{form7} \mu(\widehat{P} \cap S) < \tau \mu(\widehat{P}) \end{equation} 
by the maximality of the elements in $\mathcal{P}$. Let $\widehat{\calP}$ be the collection of maximal -- hence disjoint -- elements in $\{\widehat{P} : P \in \calP\}$. Thus each $\widehat{P} \in \widehat{\mathcal{P}}$ has at least one child in $\mathcal{P}$, and $\mu$ almost all of $S$ is still covered by $\cup \widehat{\mathcal{P}}$. Then, we have
\begin{equation}\label{form11} \sum_{\widehat{P} \in \widehat{\mathcal{P}}} \mu(\widehat{P}) \stackrel{\eqref{form7}}{>} \tau^{-1} \sum_{\widehat{P} \in \widehat{\calP}} \mu(\widehat{P} \cap S) \stackrel{\eqref{form6}}{\geq} \tau^{-1} \cdot \rho^{n} \mu(Q) = \left(\frac{\rho}{\tau}\right) \rho^{n - 1}\mu(Q). \end{equation} 
Since the cubes in $\calD$ have small boundary regions, recall \nref{D4} the "$\epsilon$-interior"
\begin{displaymath} \mathrm{int}_{\epsilon} \, \widehat{P} := \{x \in \widehat{P} : \dist(x,X \setminus \widehat{P}) \geq \epsilon \ell(\widehat{P})\} \end{displaymath} 
satisfies
\begin{displaymath} \mu(\mathrm{int}_{\epsilon} \, \widehat{P}) \geq \left(\frac{\tau}{\rho}\right)\mu(\widehat{P}), \qquad \widehat{P} \in \widehat{\mathcal{P}}, \end{displaymath} 
if $\epsilon > 0$ is chosen small enough, depending only on $\rho,\tau$, and the constants in \nref{D4}. Therefore, setting
\begin{equation}\label{form10} \widehat{S} := \bigcup_{\widehat{P} \in \widehat{\calP}} \mathrm{int}_{\epsilon} \, \widehat{P}, \end{equation}
it follows from \eqref{form11} that
\begin{displaymath} \mu(\widehat{S}) \geq \rho^{n - 1}\mu(Q). \end{displaymath} 
Consequently, by the inductive hypothesis, there exists a continuum $\widehat{\gamma} \subset X$ with $\calH^{1}(\widehat{\gamma}) \leq \ell(Q)$ and $\calH_{\infty}^{1}(\widehat{S} \cap \widehat{\gamma}) \geq (c\theta)^{n - 1} \ell(Q)$. At this point we discard -- without change in notation -- from $\widehat{\mathcal{P}}$ all the cubes with 
\begin{displaymath} \mathrm{int}_{\epsilon} \, \widehat{P} \cap \widehat{\gamma} = \emptyset. \end{displaymath}
For those cubes $\widehat{P}$ remaining, we pick an arbitrary point $c(\widehat{P}) \in \mathrm{int}_{\epsilon} \, \widehat{P} \cap \widehat{\gamma}$. We now dispose of the (special) case where $\card \widehat{\calP} = 1$, that is, $\widehat{\gamma}$ only meets the $\epsilon$-interior of a single cube $\widehat{P} \in \widehat{\mathcal{P}}$. Then $\widehat{P}$ is a cover for $\widehat{S} \cap \widehat{\gamma}$, hence
\begin{displaymath} \ell(\widehat{P}) \gtrsim \calH^{1}_{\infty}(\widehat{S} \cap \widehat{\gamma}) \geq (c\theta)^{n - 1}\ell(Q). \end{displaymath} 
Now, pick any cube $P \in \mathcal{P}$ whose parent is $\widehat{P}$. By \eqref{form8} and Proposition \ref{prop1}, there exists a curve $\gamma \subset P \subset X$ with $\calH^{1}(\gamma) \leq \ell(P) \leq \ell(Q)$ and 
\begin{displaymath} \calH^{1}_{\infty}(S \cap \gamma) \geq \theta \ell(P) \gtrsim \theta \ell(\widehat{P}) \geq c^{n - 1}\theta^{n}\ell(Q). \end{displaymath}
In particular, if $c > 0$ was small enough, $\calH^{1}_{\infty}(S \cap \gamma) \geq (c\theta)^{n}\ell(Q)$, and the proof of \eqref{form15} is complete in this case. From now on, we may then assume that $\widehat{S} \cap \widehat{\gamma}$ is not contained in any one cube $\widehat{P} \in \widehat{\mathcal{P}}$. Thus, for any $\widehat{P} \in \widehat{\mathcal{P}}$, the curve $\widehat{\gamma}$ meets both $c(\widehat{P}) \in \mathrm{int}_{\epsilon} \, \widehat{P}$ and $X \setminus \widehat{P}$. This implies that
\begin{equation}\label{form9} \calH^{1}(\widehat{\gamma} \cap \widehat{P}) \geq \epsilon\ell(P), \qquad \widehat{P} \in \widehat{\mathcal{P}}. \end{equation} 
We now construct the continuum $\gamma \subset Q$. For every $\widehat{P} \in \widehat{\calP}$, pick one cube $P \in \mathcal{P}$ whose parent is $\widehat{P}$. By \eqref{form8} and Proposition \ref{prop1}, for every such cube $P \in \calP$, there exists a curve $\gamma_{P} \subset P \subset X$ satisfying $\calH^{1}(\gamma_{P}) \leq \ell(P)$ and $\calH_{\infty}^{1}(\gamma_{P} \cap S \cap P) \geq \theta \ell(P)$. We further connect every $\gamma_{P}$ thus obtained to $c(\widehat{P}) \in \widehat{\gamma} \cap \widehat{P}$ by a continuum $\gamma_{P \to 0} \subset X$ of length $\calH^{1}(\gamma_{P \to 0}) \lesssim \ell(P)$ (using the quasiconvexity of $X$). Then, the union
\begin{displaymath} \gamma := \widehat{\gamma} \cup \bigcup_{\widehat{P} \in \widehat{\calP}} [\gamma_{P} \cup \gamma_{P \to 0}] \subset X \end{displaymath} 
is a continuum of length
\begin{align} \calH^{1}(\gamma) & \lesssim \calH^{1}(\widehat{\gamma}) + \sum_{\widehat{P} \in \widehat{\calP}} \ell(P)\\
&\label{form19} \stackrel{\eqref{form9}}{\lesssim} \calH^{1}(\widehat{\gamma}) + \sum_{\widehat{P} \in \widehat{\calP}} \calH^{1}(\widehat{\gamma} \cap \widehat{P}) \sim \calH^{1}(\widehat{\gamma}) \leq \ell(Q), \end{align}
recalling that the cubes in $\widehat{P}$ are disjoint.

We next estimate from below $\calH^{1}_{\infty}(\gamma \cap S)$. Let $\mathcal{R}$ be an arbitrary cover of $\gamma \cap S$ by disjoint cubes in $\calD$. There are two kinds of cubes in $\calR$: those which are contained in some cube of $\widehat{\mathcal{P}}$, and those which are not; we denote these families by $\mathcal{R}_{in}$ and $\mathcal{R}_{out}$. Further, $\mathcal{R}_{in}$ can be written as
\begin{displaymath} \mathcal{R}_{in} = \bigcup_{\widehat{P} \in \widehat{\mathcal{P}}} \mathcal{R}(\widehat{P}), \end{displaymath}
where $\calR(\widehat{P}) := \{R \in \calR : R \subset \widehat{P}\}$. If $\widehat{P} \in \widehat{\mathcal{P}}$ is not strictly contained in some cube of $\mathcal{R}$, then $\mathcal{R}(\widehat{P})$ is a cover of $\gamma \cap S \cap \widehat{P} \supset \gamma_{P} \cap S \cap P$, hence
\begin{displaymath} \sum_{R \in \mathcal{R}(\widehat{P})} \ell(R) \geq \calH^{1}_{\infty}(\gamma_{P} \cap S \cap P) \geq \theta \ell(P) \gtrsim \theta \ell(\widehat{P}). \end{displaymath}
Writing $\widehat{\mathcal{P}}_{0}$ for those cubes \textbf{not} strictly contained in any cube of $\mathcal{R}$, it follows that
\begin{equation}\label{form16} \sum_{R \in \mathcal{R}} \ell(R) \gtrsim \theta \left[ \sum_{R \in \mathcal{R}_{out}} \ell(R) + \sum_{\widehat{P} \in \widehat{\mathcal{P}}_{0}} \ell(\widehat{P}) \right]. \end{equation} 
To proceed, note that $\mathcal{R}_{out} \cup \widehat{\mathcal{P}}_{0}$ is a cover of $\widehat{S} \supset \widehat{S} \cap \widehat{\gamma}$, because $\widehat{\mathcal{P}}$ is a cover of $\widehat{S}$, and every cube in $\widehat{\mathcal{P}}$ is contained in some element of $\mathcal{R}_{out} \cup \widehat{\mathcal{P}}_{0}$. Consequently,
\begin{equation}\label{form17} \sum_{R \in \mathcal{R}_{out}} \ell(R) + \sum_{\widehat{P} \in \widehat{\mathcal{P}}_{0}} \ell(\widehat{P}) \gtrsim \calH^{1}_{\infty}(\widehat{S} \cap \widehat{\gamma}) \geq (c\theta)^{n - 1}\ell(Q). \end{equation}
Combining \eqref{form16}-\eqref{form17} shows that $\mathcal{H}^{1}_{\infty}(\gamma \cap B) \gtrsim c^{n - 1}\theta^{n}\ell(Q)$. The proof is now almost complete, except that the continuum $\gamma$ may be a constant times too long, compare the estimate \eqref{form19} with the goal in \eqref{form15}. However, by the subadditivity of $\calH^{1}_{\infty}$, we may choose (recall the discussion after Proposition \ref{parametrisation}) a continuum $\gamma' \subset \gamma$ which satisfies both
\begin{displaymath} \calH^{1}(\gamma') \leq \ell(Q) \quad \text{and} \quad \calH^{1}_{\infty}(\gamma' \cap B) \gtrsim c^{n - 1}\theta^{n}\ell(Q). \end{displaymath} 
Finally, if $c > 0$ was chosen small enough to begin with, we have $\calH^{1}_{\infty}(\gamma' \cap B) \geq (c\theta)^{n}\ell(Q)$, and the proof is complete. \end{proof}

For technical reasons related to the construction of $1$-rectifiable representations, in the next section, we record an alternative version of the previous proposition:
\begin{proposition}\label{mainProp2} Let $(X,d,\mu)$ be a complete, doubling, and quasiconvex $A_{\infty}(\Gamma)$-space. Then, there exist constants $C \geq 1$ and $1 \leq q < \infty$ such that the following holds. Let $Q \in \calD$, and let $S \subset Q$ be a compact set. Then, there exists a continuum $\gamma \subset X$ of length $\calH^{1}(\gamma) \leq \ell(Q)$ and a Borel measure $\nu$ supported on $S \cap \gamma$ satisfying $\nu(B(x,r)) \leq r$ for all balls $B(x,r) \subset X$, and
\begin{displaymath} \frac{\mu(S)}{\mu(Q)} \leq C\left( \frac{\nu(S \cap \gamma)}{\ell(Q)} \right)^{1/q}. \end{displaymath}
\end{proposition}
\begin{proof} Combine the previous proposition with Frostman's lemma, see \cite[Theorem 8.8]{MR1333890}. The proof given in this reference works with negligible modifications in doubling and complete metric measure spaces (which in particular support a system of dyadic cubes), and shows that if $K \subset X$ is a compact set, then there exists a Borel measure $\nu$ supported on $K$ with $\nu(K) \gtrsim \calH^{1}_{\infty}(K)$, and satisfying $\nu(B(x,r)) \leq r$ for all balls $B(x,r) \subset X$. \end{proof}

Positive Borel measures $\nu$ satisfying $\nu(B(x,r)) \leq r$ for all balls $B(x,r) \subset X$ will be called \emph{$1$-Frostman measures} in the sequel. As the last result of this section, we upgrade the previous statement to cover all Borel functions (in contrast to just characteristic functions). This can be achieved at the cost of slightly increasing the exponent $q$. 

\begin{cor}\label{mainCor} Let $(X,d,\mu)$ be a complete, doubling, and quasiconvex $A_{\infty}(\Gamma)$-space. Then, there exist constants $C \geq 1$ and $1 \leq q < \infty$ such that the following holds. Let $x \in X$, $0 < r \leq \diam(X)$, and let $\psi \colon X \to \R$ be a bounded Borel function. Then, there exists a continuum $\gamma \subset X$ of length $\calH^{1}(\gamma) \leq r$, and a $1$-Frostman measure $\nu$ supported on $\gamma$ such that
\begin{equation}\label{form18} \frac{1}{\mu(B(x,r))} \int_{B(x,r)} |\psi| \, d\mu \leq C\left(\frac{1}{r} \int_{\gamma} |\psi|^{q} \, d\nu \right)^{1/q}.  \end{equation} 
\end{cor}

\begin{proof} For every $x \in X$, $0 < r \leq \diam(X)$, and bounded Borel function $\psi \colon X \to \R$, there exists a cube $Q \in \calD$ with $r \lesssim \ell(Q) \leq r$ and
\begin{displaymath} \frac{1}{\mu(B(x,r))} \int_{B(x,r)} |\psi| \, d\mu \lesssim \frac{1}{\mu(Q)} \int_{Q} |\psi| \, d\mu. \end{displaymath} 
Therefore, in place of \eqref{form18}, it suffices to find $\gamma \subset X$ of length $\calH^{1}(\gamma) \leq \ell(Q)$, and a $1$-Frostman measure $\nu$ supported on $\gamma$, such that
\begin{equation}\label{form18a} \frac{1}{\mu(Q)} \int_{Q} |\psi| \, d\mu \leq C\left(\frac{1}{\ell(Q)} \int |\psi|^{q} \, d\nu \right)^{1/q}. \end{equation}
In proving \eqref{form18a}, we may further assume that
\begin{displaymath} \delta := \frac{1}{\mu(Q)} \int_{Q} |\psi| \, d\mu = 1, \end{displaymath} 
since otherwise we may first prove the \eqref{form18a} for the function $\varphi := \psi/\delta$ instead, and note that both sides of \eqref{form18a} have the same homogeneity. Then, 
\begin{align*} 1 & = \frac{1}{\mu(Q)} \int_{0}^{\infty} \mu(\{x \in Q : |\psi| \geq \lambda\}) \, d\lambda\\
& = \sum_{j \in \Z} \int_{2^{j - 1}}^{2^{j}} \frac{\mu(\{x \in Q : |\psi| \geq \lambda\})}{\mu(Q)} \, d\lambda\\
& \leq \sum_{j \in \Z} 2^{j} \left( \frac{\mu(\{x \in Q : |\psi| \geq 2^{j}\})}{\mu(Q)} \right). \end{align*} 
Note that the sum over the integers $j \leq -2$ is at most $1/2$, so  
\begin{displaymath} \sum_{j \geq -1}  2^{j} \left( \frac{\mu(\{x \in Q : |\psi| \geq 2^{j}\})}{\mu(Q)} \right) \geq \frac{1}{2}. \end{displaymath} 
By the pigeonhole principle, it is therefore possible to choose $j \geq -1$ and a compact set $S \subset \{x \in Q : |\psi| \geq 2^{j}\}$ such that
\begin{displaymath} \frac{\mu(S)}{\mu(Q)} \gtrsim \frac{2^{-j}}{(2 + j)^{2}}, \end{displaymath} 
Then, let $C \geq 1$ and $q_{0} \geq 1$ be the parameters appearing in Proposition \ref{mainProp2}, and let $q > q_{0}$. The proposition states that there exists a continuum $\gamma \subset X$ of length $\calH^{1}(\gamma) \leq \ell(Q)$ and a $1$-Frostman measure $\nu$ supported on $S \cap \gamma$ satisfying
\begin{displaymath} \frac{\nu(S \cap \gamma)}{\ell(Q)} \geq  \left(\frac{1}{C}\frac{\mu(S)}{\mu(Q)} \right)^{q_{0}} \gtrsim \left(\frac{2^{-j}}{(2 + j)^{2}} \right)^{q_{0}}. \end{displaymath} 
Recalling that $|\psi| \geq 2^{j}$ on $S \supset \spt \nu$, we obtain
\begin{displaymath} \left(\frac{1}{\ell(Q)} \int_{\gamma} |\psi|^{q} \, d\nu \right)^{1/q} \gtrsim 2^{j} \left(\frac{2^{-j}}{(2 + j)^{2}} \right)^{q_{0}/q} \gtrsim 1 = \frac{1}{\mu(Q)} \int_{Q} |\psi| \, d\mu, \end{displaymath} 
which completes the proof of the corollary. \end{proof}

\section{Constructing $1$-rectifiable representations}\label{s:representations}

We move towards the proof of Theorem \ref{main}. The notation $(Y,d)$ will be reserved for a compact metric space, and $\mathcal{M}(Y)$ stands for the vector space of complex Borel measures on $Y$.

\begin{definition} Let $r > 0$, and let $(Y,d)$ be a compact metric space. We denote by $\mathcal{F}_{r}(Y)$ the set of $1$-Frostman measures supported on a continuum of length $\leq r$ (recall that a $1$-Frostman measure is a positive Borel measure $\nu$ satisfying $\nu(B(x,s)) \leq s$ for all (open) balls $B(x,s) \subset Y$). 
\end{definition}

\begin{lemma}\label{compactnessLemma} If $(Y,d)$ is compact, then $\mathcal{F}_{r}(Y)$ is compact in the weak* topology. \end{lemma}

\begin{proof} Let $\{\nu_{j}\}_{j \in \N} \subset \mathcal{F}_{r}(Y)$ be an arbitrary sequence, and associate to each $\nu_{j}$ a continuum $\gamma_{j} \subset Y$ of length $\calH^{1}(\gamma_{j}) \leq r$. Then, after passing to subsequences, we may assume that $\nu_{j} \rightharpoonup \nu \in \mathcal{M}(Y)$ (here "$\rightharpoonup$" stands for weak* convergence) and $\gamma_{j} \to \gamma$ in the Hausdorff metric. Then, $\gamma$ is also a continuum satisfying $\calH^{1}(\gamma) \leq r$ by Go\l ab's theorem, see \cite[Theorem 4.4.17]{MR2012736} or \cite[Theorem 2.9]{MR3614660}. Moreover, $\nu$ is clearly supported on $\gamma$, and $\nu(B(x,s)) \leq \liminf_{j \to \infty} \nu_{j}(B(x,s)) \leq s$ for all balls $B(x,s) \subset X$ by weak* convergence.  \end{proof}

We repeat the definition of $1$-rectifiable representations for the reader's convenience. 

\begin{definition} Let $(Y,d)$ be a compact metric space, let $r > 0$, and let $\mathcal{P}_{r} := \mathcal{P}(\mathcal{F}_{r}(Y))$ be the set of Borel probability measures on the compact space $\mathcal{F}_{r}(Y)$, defined above. Let $\nu_{\tn}$ be the measure
\begin{displaymath} \nu_{\tn} := \frac{1}{r} \int \nu \, d\tn(\nu) \in \mathcal{M}(Y). \end{displaymath} 
More formally, $\nu_{\tn}$ is the measure given by the Riesz representation theorem applied to the linear functional
\begin{displaymath} \varphi \mapsto \frac{1}{r} \int \left( \int \varphi \, d\nu \right) \, d\tn(\nu), \qquad \varphi \in C(Y). \end{displaymath}
The right hand side above is well-defined, because $\nu \mapsto \int \varphi \, d\nu$ is continuous in the weak* topology, and $\tn$ was assumed to be a Borel measure on $\mathcal{F}_{r}$. A measure $\mu \in \mathcal{M}(Y)$ has a \emph{$1$-rectifiable representation of length $r$ in $L^{p}$} if there exists $\tn \in \mathcal{P}_{r}$ such that $\mu \ll \nu_{\tn}$ with $\mu \in L^{p}(\nu_{\tn})$. In this case, we write
\begin{displaymath} \|\mu\|_{L^{p}(r)} := \inf\{\|\mu\|_{L^{p}(\nu_{\tn})} : \tn \in \mathcal{P}\} < \infty. \end{displaymath} 
\end{definition}

\begin{proposition} Let $1 < p \leq \infty$, let $(Y,d)$ be compact, and let $r > 0$. Then, the set
\begin{equation}\label{defN} \mathcal{N} := \mathcal{N}_{p}(r) := \left\{\mu \in \mathcal{M}(Y) : \|\mu\|_{L^{p}(r)} \leq 1 \right\} \end{equation}
is convex, balanced, and closed, hence compact in the weak* topology of $\mathcal{M}(Y)$. \end{proposition}
\begin{proof} It is clear that $\mathcal{N}$ is balanced: if $\mu \in \mathcal{N}$ and $\alpha \in \C$ with $|\alpha| \leq 1$, then $\|\alpha \mu\|_{L^{p}(\nu_{\tn})} \leq \|\mu\|_{L^{p}(\nu_{\tn})}$ for any $\tn \in \mathcal{N}$. We next prove that $\mathcal{N}$ is closed in $\mathcal{M}(Y)$. The same argument will also reveal that to every $\mu \in \mathcal{N}$ there corresponds some $\tn \in \mathcal{P}_{r} := \mathcal{P}(\mathcal{F}_{r}(Y))$ such that $\|\mu\|_{L^{p}(\nu_{\tn})} \leq 1$. Pick a sequence $\{\mu_{j}\}_{j \in \N} \subset \mathcal{N}$ such that
\begin{displaymath} \mu_{j} \rightharpoonup \mu \in \mathcal{M}(Y) \end{displaymath}
and associate to each $\mu_{j}$ a probability $\tn_{j} \in \mathcal{P}$ such that
\begin{equation}\label{form27} \|\mu_{j}\|_{L^{p}(\nu_{\tn_{j}})} \leq 1 + \frac{1}{j}. \end{equation}
Since $\mathcal{F}_{r}(Y)$ is compact by Lemma \ref{compactnessLemma}, also $\mathcal{P}_{r}$ is compact in the weak* topology, and hence we may assume that $\tn_{j} \rightharpoonup \tn \in \mathcal{P}_{r}$ after passing to a subsequence. It remains to show that $\mu \in L^{p}(\nu_{\tn})$ with $\|\mu\|_{L^{p}(\nu_{\tn})} \leq 1$. We use duality: fix $\psi \in C(Y)$. Then, by H\"older's inequality and \eqref{form27}, we first have
\begin{displaymath} \left| \int \psi \, d\mu \right| \leq \limsup_{j \to \infty} \left| \int \psi \, d\mu_{j} \right| \leq \limsup_{j \to \infty} \left(\int |\psi|^{q} \, d\nu_{\tn_{j}} \right)^{1/q}. \end{displaymath}
Here $1/p + 1/q = 1$ (with $q = 1$ if $p = \infty$). Writing $\varphi := |\psi|^{q} \in C(Y)$, we further have
\begin{displaymath} \int \varphi \, d\nu_{\tn_{j}} = \frac{1}{r} \int \int \varphi \, d\nu \, d\tn_{j}(\nu) \to \frac{1}{r} \int \int \varphi \, d\nu \, d\tn(\nu) = \int \varphi \, d\nu_{\tn}, \end{displaymath} 
because the map $\nu \mapsto \int \varphi \, \nu$ is continuous $\mathcal{F}_{r}(Y) \to \R$ (where $\mathcal{F}_{r}(Y)$ is equipped with the weak* topology, as always). Thus,
\begin{displaymath} \left|\int \psi \, d\mu \right| \leq \left(\int |\psi|^{q} \, d\nu_{\tn} \right)^{1/q}, \end{displaymath} 
which shows that $\mu \in L^{p}(\nu_{\tn})$ with $\|\mu\|_{L^{p}(\nu_{\tn})} \leq 1$. Hence, $\mu \in \mathcal{N}$, and it has been established that $\mathcal{N}$ is closed.

It remains to prove the convexity of $\mathcal{N}$. Let $\mu_{1},\mu_{2} \in \mathcal{N}$, and let $\tn_{1},\tn_{2} \in \mathcal{P}_{r}$ be such that
\begin{equation}\label{form20} \|\mu_{1}\|_{L^{p}(\nu_{\tn_{1}})} \leq 1 \quad \text{and} \quad \|\mu_{2}\|_{L^{p}(\nu_{\tn_{2}})} \leq 1. \end{equation} 
Let $\lambda_{1},\lambda_{2} \in [0,1]$ with $\lambda_{1} + \lambda_{2} = 1$, and consider the probability $\tn := \lambda_{1}\tn_{1} + \lambda_{2}\tn_{2} \in \mathcal{P}_{r}$. Note that
\begin{equation}\label{form21} \lambda_{1}\nu_{\tn_{1}} + \lambda_{2}\nu_{\tn_{2}} = \nu_{\tn}. \end{equation}
To show that $\lambda_{1}\mu_{1} + \lambda_{2}\mu_{2} \in L^{p}(\nu_{\tn})$, we again employ duality. Fix $\psi \in C(Y)$, and let $1/p + 1/q = 1$ (with $q = 1$ if $p = \infty$). Then, using H\"older's inequality and the concavity of $t \mapsto t^{1/q}$, we infer that
\begin{align*} \left| \int \psi \, d[\lambda_{1}\mu_{1} + \lambda_{2}\mu_{2}] \right| & \leq \lambda_{1} \left| \int \psi \, d\mu_{1} \right| + \lambda_{2} \left| \int \psi \, d\mu_{2} \right| \\
& \stackrel{\eqref{form20}}{\leq} \lambda_{1} \left( \int |\psi|^{q} \, d\nu_{\tn_{1}} \right)^{1/q} + \lambda_{2} \left( \int |\psi|^{q} \, d\nu_{\tn_{2}} \right)^{1/q}\\
& \leq \left(\int |\psi|^{q} d[\lambda_{1} \nu_{\tn_{1}} + \lambda_{2} \nu_{\tn_{2}}] \right)^{1/q} \stackrel{\eqref{form21}}{=} \left(\int |\psi|^{q} \, d\nu_{\tn} \right)^{1/q}. \end{align*}
This proves that $\lambda_{1}\mu_{1} + \lambda_{2}\mu_{2} \in L^{p}(\nu_{\tn})$ with $\|\lambda_{1}\mu_{1} + \lambda_{2}\mu_{2}\|_{L^{p}(\nu_{\tn})} \leq 1$, and consequently $\lambda_{1}\mu_{1} + \lambda_{2}\mu_{2} \in \mathcal{N}$. \end{proof}

The next proposition gives a criterion for a measure to belong to $\mathcal{N}$; compare this with Corollary \ref{mainCor} to see where we are headed. 

\begin{proposition}\label{criterion} Let $1 \leq q < \infty$, let $r > 0$, and let $(Y,d)$ be a compact metric space. Assume that $\mu \in \mathcal{M}(Y)$ is a measure satisfying
\begin{equation}\label{form3} \|\psi\|_{L^{1}(\mu)} \leq \sup_{\tn \in \mathcal{P}_{r}} \|\psi\|_{L^{q}(\nu_{\tn})}, \qquad \psi \in C(Y), \end{equation} 
where $\calP_{r} := \calP(\mathcal{F}_{r}(Y))$. Then $\mu \in \mathcal{N}_{p}(r)$, where $1/p + 1/q = 1$ (and $p = \infty$ if $q = 1$).
\end{proposition}

\begin{proof} Assume to the contrary that $\mu \notin \mathcal{N}_{p}(r) =: \mathcal{N}$. The dual space of $\mathcal{M}(Y)$ equipped with the weak* topology is $C(Y)$ (because the weak* topology is by definition the coarsest topology which makes the functionals $\mu \mapsto \int \psi \, d\mu$, for $\psi \in C(Y)$, continuous; see also \cite[Theorem 3.10]{MR1157815}). Thus, because $\mathcal{N} \subset \mathcal{M}(Y)$ is convex, balanced, and compact, there exists by \cite[Theorem 3.7]{MR1157815} a number $\tau < 1$ and a function $\psi \in C(Y)$ such that
\begin{equation}\label{form22} \sup_{\nu \in \mathcal{N}} \left| \int \psi \, d\nu \right| < \tau \cdot \int \psi \, d\mu. \end{equation} 
By \eqref{form3}, we then infer that there exists $\tn \in \mathcal{P}_{r}$ such that
\begin{equation}\label{form2} \tau \cdot \|\psi\|_{L^{1}(\mu)} \leq \|\psi\|_{L^{q}(\nu_{\tn})}. \end{equation} 
There are now at least two ways to complete the proof, each of them so short that we record both. In the first argument, choose (by duality) a function $\rho \in L^{p}(\nu_{\tn})$ with $\|\rho\|_{L^{p}(\nu_{\tn})} = 1$ such that
\begin{equation}\label{form4} \left| \int \psi \rho \, d\nu_{\tn} \right| = \|\psi\|_{L^{q}(\nu_{\tn})}. \end{equation} 
Note that $\rho \, d\nu_{\tn} \in \mathcal{N}_{p}(r)$ by definition. It follows that
\begin{displaymath} \|\psi\|_{L^{q}(\nu_{\tn})} \stackrel{\eqref{form4}}{=} \left| \int \psi \rho \, d\nu_{\tn} \right| \stackrel{\eqref{form22}}{<} \tau \cdot \int \psi \, d\mu \stackrel{\eqref{form2}}{\leq} \|\psi\|_{L^{q}(\nu_{\tn})}.   \end{displaymath} 
This contradiction completes the proof.

For the second argument, let $\nu_{\tn}$ be as in \eqref{form2}, and define the norm
\begin{displaymath} \varphi \mapsto p(\varphi) := \left( \int |\psi|^{q} \, d\nu_{\tn} \right)^{1/q}, \qquad \varphi \in C(Y). \end{displaymath}
Then \eqref{form2} implies that 
\begin{displaymath} \left|\int \lambda \psi \, d\mu \right| \leq p(\lambda \psi), \qquad \lambda \geq 0, \end{displaymath} 
which means that the linear functional $\varphi \mapsto \int \varphi \, d\mu$ is dominated by $p$ on the $1$-dimensional subspace $\spa(\psi) \subset C(Y)$. Therefore the Hahn-Banach and Riesz representation theorems (see \cite[Theorem 3.3]{MR1157815}) give a measure $\nu \in \mathcal{M}(Y)$ which agrees with $\mu$ on $\psi$ and is dominated by $p$ everywhere on $C(Y)$. This implies by duality that $\nu \in \mathcal{N}_{p}(r)$, whence
\begin{displaymath} \int \psi \, d\mu = \int \psi \, d\nu \stackrel{\eqref{form22}}{<} \int \psi \, d\mu.  \end{displaymath}
Again, a contradiction has been reached. \end{proof}

Now we have all the ingredients to prove Theorem \ref{main}, which we repeat here for convenience:
\begin{thm} Let $(X,d,\mu)$ be a complete, doubling, and quasiconvex space which is $A_{\infty}$ on curves. Then, there exist constants $p > 1$ and $A \geq 1$ such that the following holds. For any $x \in X$ and $0 < r \leq \diam(X)$, the normalised restriction 
\begin{displaymath} \mu_{x,r} := \frac{\mu\llcorner_{B(x,r)}}{\mu(B(x,r))} \end{displaymath}
has a $1$-rectifiable representation of length $r$ in $L^{p}$, and moreover $\|\mu_{x,r}\|_{L^{p}(r)} \leq A$. \end{thm}

\begin{proof} Let $1 \leq q < \infty$ and $C \geq 1$ be the parameters from Corollary \ref{mainCor}. We verify the criterion in Proposition \ref{criterion} applied to the measure $\tilde{\mu}_{x,r} := C^{-(1 + 1/q)} \cdot \mu_{x,r}$ and the compact metric space $Y = \overline{B(x,2r)}$. Let $\psi \in C(Y)$. If
\begin{displaymath} \|\psi\|_{L^{1}(\tilde{\mu}_{x,r})} = 0, \end{displaymath}
condition \eqref{form3} is trivial. Otherwise, apply Corollary \ref{mainCor} to the bounded Borel function $\psi \mathbf{1}_{B(x,r)}$. The conclusion is that there exists a continuum of length $\leq r$ and a $1$-Frostman measure $\nu$ supported on $\gamma$ such that
\begin{equation}\label{form24} 0 < \|\psi\|_{L^{1}(\tilde{\mu}_{x,r})} \leq \left(\frac{1}{Cr} \int_{\gamma \cap B(x,r)} |\psi|^{q} \, d\nu \right)^{1/q}. \end{equation}
Clearly $\gamma \cap B(x,r) \neq \emptyset$, which implies that $\gamma \subset Y$, and hence $\nu \in \mathcal{F}_{r}(Y)$. Therefore, the Dirac mass on $\gamma$, denoted $\delta(\gamma)$, is an element in $\calP := \mathcal{P}(\mathcal{F}_{r}(Y))$, and the inequality \eqref{form24} states that
\begin{displaymath} \|\psi\|_{L^{1}(\tilde{\mu}_{x,r})} \leq \|\psi\|_{L^{q}(\nu_{\delta(\gamma)})} \leq \sup_{\tn \in \mathcal{P}} \|\psi\|_{L^{q}(\nu_{\tn})}. \end{displaymath} 
We may now infer from Proposition \ref{criterion} that $\tilde{\mu}_{x,r} \in \mathcal{N}_{p}(r)$. It follows that
\begin{displaymath} \|\mu_{x,r}\|_{L^{p}(r)} \leq C^{1 + 1/q}. \end{displaymath} 
This easily implies the theorem with $A := C^{1 + 1/q}$.
\end{proof}

\bibliographystyle{plain}
\bibliography{references}

\end{document}